\documentclass[12pt]{amsart}

\usepackage[utf8]{inputenc}
\usepackage[T1]{fontenc}
\usepackage[final]{microtype}
\usepackage{hyperref,amsthm,amssymb,interval,tikz,tikz-cd}
\usepackage{amsmath,amsfonts,amssymb,amsthm,graphicx,color,srcltx,enumitem,bm,cancel,doi,xcolor}
\usepackage[osf]{Baskervaldx}
\usepackage[frenchmath,baskervaldx]{newtxmath}
\usepackage[default]{frcursive}
\usepackage[T1]{fontenc}
\usepackage{enumitem}

\newcommand{\Z}{\mathbb{Z}}

\newcommand{\D}{\mathbb{D}}
\newcommand{\A}{\mathbb{A}}
\newcommand{\G}{\mathbb{G}}

\newcommand{\logg}{\mathsf{Log}}
\newcommand{\dimm}{\mathsf{dim}}

\newcommand{\spec}{\mathsf{Spec}\,}

\newcommand{\hcr}{\widetilde{\mathscr{H}(x)}}

\usepackage{mathtools}
\usepackage{mathrsfs}
\usepackage{xcolor}
\theoremstyle{definition}
\newtheorem{example}[equation]{Example}

\newtheorem{situation}[equation]{}

\newtheorem{remark}[equation]{Remark}
\newtheorem{definition}[equation]{Definition}
\theoremstyle{plain}
\newtheorem{theorem}[equation]{Theorem}

\newtheorem*{theorem*}{Theorem}
\newtheorem{proposition}[equation]{Proposition}
\newtheorem{lemma}[equation]{Lemma}
\newtheorem{conjecture}[equation]{Conjecture}

\newcommand{\defeq}{:=}
\numberwithin{equation}{section}
\usepackage[]{enumitem}
\setlist[enumerate]{labelindent=\parindent,leftmargin=*,topsep=0.4ex,itemsep=0.1ex}
\setlist[itemize]{labelindent=\parindent,leftmargin=*,topsep=0.4ex,itemsep=-1ex,partopsep=1ex,parsep=1ex}
\setlist[enumerate,1]{labelindent=\parindent, leftmargin=*,label=\textup{(\arabic*)},ref=\textup{\arabic*}}

\renewcommand{\setminus}{\mathbin{\rule[0.2em]{0.67em}{0.12em}}}%

\providecommand{\keywords}[1]{\textbf{{Keywords---}} #1}

\renewcommand{\mathbb}{\mathbf}

\usepackage[paperwidth=175mm,paperheight=257mm,text={130mm,210mm},centering]{geometry}
\usepackage{cjhebrew}
\makeatletter
\newsavebox{\@brx}
\newcommand{\llangle}[1][]{\savebox{\@brx}{\(\m@th{#1\langle}\)}%
  \mathopen{\copy\@brx\kern-0.5\wd\@brx\usebox{\@brx}}}
\newcommand{\rrangle}[1][]{\savebox{\@brx}{\(\m@th{#1\rangle}\)}%
  \mathclose{\copy\@brx\kern-0.5\wd\@brx\usebox{\@brx}}}
\makeatother

\title{On the Structure of the Complement of Skeleton}

\begin{document}
\author{Morgan Brown}
\address{Department of Mathematics, University of Miami, Ungar Bldg, 1365 Memorial Dr \#515, Coral Gables, FL 33146, USA.}
\email{mvbrown@math.miami.edu}


\author{$\text{JIACHANG 
 XU}^{\dagger}$}\thanks{$\dagger$: Corresponding author}
\address{Institute of Mathematics and Informatics, Bulgarian Academy of Sciences, Bulgaria, Sofia 1113, Acad. G. Bonchev Str., Bl. 8}
\email{jiachangxu823@gmail.com}

\author{Muyuan Zhang}
\address{Westlake Institute for Advanced Study; Institute for Theoretical Sciences, Westlake University, Hangzhou 310030, China}
\email{mzhang73@outlook.com}

\begin{abstract}
We study the higher dimensional geometry of Berkovich spaces using virtual open disks, which are given by fibration of relative dimension $1$. Inspired by birational geometry, we conjecture that the Berkovich skeleton is the complement of the union of all virtual open disks, and prove this conjecture for $\mathcal{X}$ admitting a strictly semistable model with semiample canonical class.
\end{abstract}

\maketitle

\section{Introduction}








     Throughout the paper, we fix a non-Archimedean field $K$ with valuation ring $K^{\circ}$, its residue field $\widetilde{K}$, and a pseudo-uniformizer $\varpi_{K}$. The purpose of this article is to study the structure of the complement of skeleton of an analytic space in the sense of Berkovich \cite{MR1070709}. The development of the related study started from \cite{MR0774362} by Bosch and L\"utkebohmert for semi-stable reduction of curves in the rigid geometry setting, then by \cite{MR2395137} and \cite{MR3204269}, the authors described the structure of the complement of the Berkovich skeleton as a union of virtual open disks for Berkovich curves which admit semi-stable models. The present work aims at accomplishing two goals: First, we wish to elucidate the structure of higher dimensional Berkovich spaces; our second goal is to extend the analogy between Berkovich geometry and birational geometry developed in the study of the essential skeleton, which guides our approach to the first goal.

\par
\begin{situation}\textit{The Essential Skeleton}
    
The essential skeleton $\textsf{Sk}^{\textsf{ess}}(\mathcal{X}^{\textsf{an}})$ was first defined by Kontsevich and Soibelman \cite{MR2181810} to study Mirror Symmetry for K3 surfaces. Their definition was extended by Musta\c{t}\u{a} and Nicaise \cite{MR3370127} for any variety that admits global nonvanishing pluricanonical forms. The essential skeleton is a birational invariant of $\mathcal{X}$ and does not depend on the choice of a model over a discrete valuation ring of equal characteristic 0. However, should $\mathcal{X}$ admit a good dlt minimal model over a discrete valuation ring of equal characteristic 0, then the essential skeleton is identified with the Berkovich skeleton of that model \cite{MR3595497}, In certain cases, for example, the moduli space of genus zero stable curves with $n$-marked points, its associated essential skeleton also coincide with its faithful tropicalization \cite{xu2024faithfultropicalizationskeletonoverlinem0n,alma991031696915802976}.

\end{situation}

\begin{situation}\textit{Virtual Open Disks}

Mori's conjecture states that a characteristic $0$ projective variety $\mathcal{X}$ is uniruled (has a dense algebraic family of rational curves) if and only if the Kodaira dimension of $\mathcal{X}$ is $-$$\infty$. This is to say that the absence of pluricanonical forms on $X$ should imply the presence of a large family of rational curves. We introduce \textit{virtual open disks of relative dimension $1$} \ref{Open virtual disk} to be the Berkovich analog of these families. Roughly speaking, a virtual open disk in $\mathcal{X}^{\textsf{an}}$ comes from expressing $\mathcal{X}^{\textsf{an}}$ locally as a fibration of relative dimension $1$ \ref{Open virtual disk}, by finding a virtual open disk over a Berkovich point in the base, which in general will correspond to a valued field extension with some degree of transcendence of $K$.
\end{situation}
\begin{situation}\textit{Main Conjecture}

We conjecture in general that the complement of the Berkovich skeleton is a union of virtual open disks in $\mathcal{X}^{\textsf{an}}$:

\begin{conjecture}
(Conjecture \ref{bigconjecture})
 Let $\mathcal{X}$ be a smooth proper integral variety with dimension $n$ over a complete discrete valued field $K$ with equal characteristic $0$, then for any point $x \in \mathcal{X}^{\textsf{an}}$, $x$ is contained in a virtual open disk if and only if $x$ is not contained in the essential skeleton.
\end{conjecture}

We are able to prove this conjecture under the condition that $
\mathcal{X}$ admits a strictly semi-stable model with semi-ample canonical class. Although this is a restrictive condition, this theorem provides evidence for the soundness of our approach.

\begin{theorem}(Main Theorem \ref{semiample canonical})
Let ${{\mathcal{X}}}$ be a proper and strictly semi-stable scheme over a complete discrete valuation ring $K^{\circ}$ such that the canonical divisor $K_\mathcal{X}$ is a semi-ample. Then for any point $x \in \mathcal{X}_{K}^{\textsf{an}}$, $x$ is contained in a virtual open disk if and only if $x$ is not contained in the essential skeleton.
\end{theorem}

\end{situation}

\begin{situation}
We prove Theorem \ref{semiample canonical} in two parts. We use the local picture of $\mathbb{G}^{n,\textsf{an}}_{m,K}$ fibered over $\mathbb{G}^{n-1,\textsf{an}}_{m,K}$ to establish the existence of open virtual disks outside the Berkovich skeleton in the semi-stable case. 
\begin{theorem}

    (Theorem \ref{main theorem 1}). Let ${{\mathcal{X}}}$ be a strictly semi-stable scheme over $K^{\circ}$, the complement of Berkovich skeleton $\textsf{Sk}({\widehat{\mathcal{X}}}_{\eta})^{\textsf{c}}$ is a union of open virtual disks.

\end{theorem}

\end{situation}



\begin{situation}
    On the other hand, we must exclude the existence of open virtual disks from intersecting the essential skeleton. This procedure is analogous to the exclusion of families of rational curves from varieties admitting a global pluricanonical form. A key technical ingredient is Temkin's work on metrization of the pluricanonical forms \cite{MR3702313}.
\begin{theorem}\label{main theorem 2}(Main Theorem \ref{points in the essential skeleton})
     Let $\mathcal{X}$ be a smooth proper integral variety with dimension $n$ over $K$, then for any point $x$ of the essential skeleton $\textsf{Sk}^{\textsf{ess}}(\mathcal{X}^{\textsf{an}})$ there exists no open virtual disk in $\mathcal{X}$ containing $x$. 
\end{theorem}
\end{situation}
\begin{situation}
When $\mathcal{X}$ itself is uniruled, we consider the essential skeleton to be empty, so we conjecture that $\mathcal{X}^{\textsf{an}}$ is covered by open virtual disks after a finite field extension.
\begin{conjecture}
    Let $\mathcal{X}$ be a uniruled variety over $K$, then $\mathcal{X}^{\textsf{an}}$ can be covered by open virtual disks after a finite field extension $K'|K$.
\end{conjecture}
By using the theorem \ref{main theorem 1}, we have the following result:
\begin{example}
    $\mathbb{P}_{K}^{n,\textsf{an}}$ can be covered by open virtual disks after a finite field extension $K'|K$.
\end{example}
\end{situation}
\begin{situation}
An alternative description for the complement of Berkovich skeleton of semi-stable variety is motivated by a simple observation of V. Berkovich \cite{MR2263704}, that is, the maximal analytic continuation of Iwasawa log function $\logg^{\lambda}$ coincides with the complement of the essential skeleton of the one-dimensional torus. By using the local model of semi-stable reduction, we generalized the observation to the following proposition:
\begin{proposition}(Proposition \ref{log description of complement of semistable skeleton})
        Let ${{\mathcal{X}}}$ be a strictly semi-stable scheme over $K^{\circ}$, the complement of Berkovich skeleton $\textsf{Sk}({\widehat{\mathcal{X}}}_{\eta})^{\textsf{c}}$ is a union of analyticity sets of the naive analytic function $\logg^{\lambda}$.
\end{proposition}
\end{situation}
\subsection*{Acknowledgements.}

The first author was supported by the Simons Foundation Collaboration Grant 524003 and he is very grateful to the Institute for Mathematics and Informatics at the Bulgarian Academy of Sciences (IMI-BAS) for hosting his visit during the completion of this project. 
Most part of this work was completed whilst the second author stayed at Mathematics and Informatics at the Bulgarian Academy of Sciences (IMI-BAS) and the second author was supported by the National Science Fund of Bulgaria, National Scientific Program "VIHREN", Project no. KP-06-DV-and he would like to thank Ludmil Katzarkov for his support. Parts of preliminary results have been dealt with in the third author's PhD dissertation.  We want to thank the anonymous referees for their valuable
suggestions.

\section{Preliminaries}
\begin{situation}
    For our purposes a \textit{non-Archimedean field} means a complete topological field $K$ with a nontrivial non-Archimedean valuation $|-|_{K}$. There exists a nonzero element $\varpi$ \textit{pseudo-uniformizer} such that $|\varpi|_{K} <1$ and $\varpi$-adic topology coincides with the valuation topology. We denote $\widehat{K^{\textsf{alg}}}$ as the completion of algebraic closure of $K$.
\end{situation}
\begin{situation}\label{K-analytic space} 
Let $K$ be a non-Archimedean field. All $K$-analytic spaces are assumed to be \textit{Hausdorff} and \textit{strictly $K$-analytic}.    
\end{situation}

\begin{situation}
    Let $\mathcal{X}$ be a locally finite type separated scheme over $K$, we denote by $\mathcal{X}^{\textsf{an}}$ the associated $K$-analytic space defined in \cite[3.4.1. Theorem]{MR1070709}. Let $\mathfrak{X}$ be a separated formal scheme locally finitely presented over the ring of integers $K^{\circ}$, we denote by $\mathfrak{X}_{\eta}$ the associated $K$-analytic space which is called the Berthelot generic fiber in the sense of \cite[§ 1]{MR1262943}. There is a canonical open immersion ${\mathfrak{X}_{\eta}}\hookrightarrow \mathcal{X}^{\textsf{an}}$ and ${\mathfrak{X}_{\eta}}=\mathcal{X}^{\textsf{an}}$ when $\mathcal{X}$ is proper.
    \end{situation}

\begin{situation}
    
        Let $L$ be a non-Archimedean $K$-field. For the algebraic extension part, we denote $e_{L|K}$\footnote{Ramification index if $L|K$ is finite.} and $f_{L|K}$ be the cardinality of the quotient of value groups $|L^{\times}|/|K^{\times}|$ and $[\widetilde{L}:\widetilde{K}]$ respectively. For the transcendental extension part, we denote $s_{L|K}\defeq \textsf{rank}_{\mathbb{Q}}(|L^{\times}|/|K^{\times}|\otimes_{\mathbb{Z}}\mathbb{Q})$\footnote{$\sqrt{|K^{\times}|}=|K^{\textsf{sep}, \times}|=\mathbb{Q}\otimes_{\Z}|K^{\times}|$},  $t_{L|K}\defeq \textsf{tr.deg.}(\widetilde{L}|\widetilde{K})$ and $d_{L|K} \defeq s_{L|K}+t_{L|K}$. For a point $x$ belonging to a $K$-analytic space $X$, we denote $d_{K}(x)$, $s_{K}(x)$ and $t_{K}(x)$ as $d_{\mathscr{H}(x)|K}$, $s_{\mathscr{H}(x)|K}$ and $t_{\mathscr{H}(x)|K}$ respectively. The set $X_{\textsf{st}} \defeq \{x \in X\,|\,d_{K}(x)=0\}$ has the following simple property under \'etale morphisms:
     \begin{lemma}
\label{points under \'etale morphisms}
        Let $f: X \longrightarrow Y$ be a morphism of $K$-analytic spaces with relative dimension $0$, then we have $f^{-1}(Y_{\textsf{st}})=X_{\textsf{st}}$.
        \begin{proof}
           Note that morphisms of $K$-analytic spaces do not increase the value of $s_{K}(x)$ and $t_{K}(x)$, so it suffices to show the $d_{K}(f(x))=d_{K}(x)$ which is equivalent to $d_{\mathscr{H}(f(x))}(\mathscr{H}(x))=0$. This holds since $\mathscr{H}(x)|\mathscr{H}(f(x))$ is finite. 
        \end{proof}      
     \end{lemma}
\end{situation}
    \begin{definition}
        Let $X$ be an $K$-analytic space, a point $x \in X$ is called \textit{monomial} if $d_{K}(x)=\dimm_{x}(X)$ and the set of monomial points of $X$ is denoted by $X^{\textsf{mon}}$. A monomial point $x \in X$ is called \textit{divisorial point} if $t_{K}(x)=\dimm_{x}(X)$.
    \end{definition}
    \begin{situation}
            For $X=\mathcal{X}^{\textsf{an}}$ when $\mathcal{X}$ is integral, we call the preimage of the generic point of canonical morphism $\tau: \mathcal{X}^{\textsf{an}} \longrightarrow \mathcal{X}$ the set of \textit{birational points} of $\mathcal{X}$ and denote by $\mathcal{X}^{\textsf{bir}}$, which can be also interpreted as the set of real valuations $\textsf{K}(X) \longrightarrow \mathbb{R}$ that extend the valuation $|-|_{K}$ on the ground field $K$. The point $x \in \mathcal{X}^{\textsf{an}}$ is divisorial if and only if there exists a formal model $\mathfrak{X}$ of $\widehat{\mathcal{X}}_{\eta}$ such that $x$ is the preimage of a generic point of $\mathfrak{X}_{\textsf{s}}$ under the reduction map. For a given semi-stable formal model $\mathfrak{X}$ of ${\widehat{\mathcal{X}}}_{\eta}$, we have associated Berkovich skeleton denoted as $\textsf{Sk}(\mathfrak{X}) \subseteq \mathcal{X}^{\textsf{mon}} \subseteq \mathcal{X}^{\textsf{bir}}$ of ${\widehat{\mathcal{X}}}_{\eta}$, for more details of constructions associated to strictly semi-stable (polystable) formal scheme, please see \cite[4.6.]{MR3479562} or \cite{MR1702143}. A divisorial point $x \in \mathcal{X}^{\textsf{an}}$ is monomial if and only if there exits a formal model of $\widehat{\mathcal{X}}_{\eta}$ such that $x \in \textsf{Sk}(\mathfrak{X})$. Throughout the paper, we will use $\textsf{Sk}^{\textsf{ess}}(\mathcal{X}^{\textsf{an}})$ to represent the essential skeleton defined in \cite[Definition 4.10.]{MR3370127} for a smooth projective integral scheme over a complete discrete valued field $K$ with equal characteristic $0$, and $\textsf{Sk}^{\textsf{k\"ah}}(X)$ to represent the  K\"ahler skeleton for a quasi-smooth $K$-analytic space $X$, see Definition \ref{Kahsk defin}. In the case of $\G^{n,\textsf{an}}_{m,K}$, one can construct its standard skeleton $\textsf{Sk}(\G^{n,\textsf{an}}_{m,K})$, see \cite{MR3702313}. Moreover, $\textsf{Sk}(\G^{n,\textsf{an}}_{m,K})$ coincides with its K\"ahler skeleton by \cite[Remark 7.2.2.]{MR3702313}.
    \end{situation}

\begin{proposition}\cite[Proposition 2.4.8.]{MR3370127}
    When the ground field $K$ is discretely valued. Let $x$ be a monomial point of $\mathcal{X}^{\textsf{an}}$. Then the following are equivalent:
    \begin{enumerate}
        \item The point $x$ is divisorial.
        \item The valuation $|-|_{x}$ is discrete.
        \item The analytic space $\mathcal{X}^{\textsf{an}}$ has rational rank one at $x$.
    \end{enumerate}
        
    \end{proposition}

    When $x$ is a divisorial point, we have $\mathscr{H}(x)$ is a complete discrete valuation field. The associated valuation $|-|_{x}$ extends $|-|_{K}$.
\par
We review the notations of discs and annuli throughout this paper:

\begin{situation}
    \textsc{Discs and annuli}\label{disks and annuli}.
    Let $a$  be a rational point of $\mathbb{A}_{K}^{1,\textsf{an}}$ and $r$, $r_{1} < r_{2}$ be the non-zero real numbers.
     A standard form of open disc is denoted by $\mathbb{D}(a,r)_{K}=\{x \in \mathbb{A}_{K}^{1,\textsf{an}}\,|\,|T-a|_{x} < r\}$ . A standard form of closed annulus is denoted by $\mathbb{B}(a; r_{1},r_{2})_{K} \defeq \{x \in \mathbb{A}_{K}^{1,\textsf{an}}\,|\, r_{1} \leq |T-a|_{x} \leq r_{2}\}$
    \par

\begin{situation}\label{Open virtual disk}\textsc{open virtual disk}.
       Let $X$ be a connected $K$-analytic curve, then $X$ is called \textit{open (closed) virtual disc or annulus} over $K$ if $X_{\widehat{K^{\textsf{alg}}}}$ is a disjoint union of open (closed) discs or annuli. We use $\mathbb{D}_{K}$ to denote an open virtual disc over a field $K$.
\end{situation}


\end{situation}

    

\section{The complement of skeleton for higher dimensional $K$-analytic space}
\subsection{Semi-stable reduction}

\begin{definition}\label{standard semi-stable model}
    Let $\mathcal{X}$ be a scheme locally of finite presentation over $K^{\circ}$. For an open subscheme $\mathcal{U}_{\mathcal{X}}$ of $\mathcal{X}$, a pair $(\mathcal{X},\mathcal{U}_{\mathcal{X}})$ is called \textit{(strictly) semi-stable} over $K^{\circ}$ 
    if there exists an (Zariski) \'etale 
covering $\mathscr{V}=\{\mathcal{V}\}$ if we have (open) \'etale  maps:
\begin{equation}
\begin{tikzcd}[column sep=3em]
& (\mathcal{V},\mathcal{U}_{\mathcal{V}}) \arrow[dl,,"f\, \textsf{open}"] \arrow[dr,"g\,\textsf{\'et}"] & \\
(\mathcal{X},\mathcal{U}_{\mathcal{X}})  &&  \big(\mathcal{M} \defeq \textsf{Spec}\,K^{\circ}{[T_{0},\dots,T_{n}]}/(T_{0} \cdots T_{r}-\varpi), \mathcal{U}
\big)\end{tikzcd}
\end{equation}
such that $f^{-1}(\mathcal{U}_{\mathcal{X}})=\mathcal{U}_{\mathcal{V}}=g^{-1}(\mathcal{U})$ and $$\mathcal{U} \cong \textsf{Spec}\,K^{\circ}[T_{0},\dots,T_{n},T^{-1}_{0},\dots,T^{-1}_{m}]/(T_{0}\cdots T_{r}-\varpi)$$ for $0 \leqslant r \leqslant m \leqslant n$ and $\varpi$ is a fixed uniformizer of the complete discrete valuation ring $K^{\circ}$.
\end{definition}

From the description above, we have the morphisms above are strict log \'etale and 
\begin{align*}
\mathcal{U} &\cong \textsf{Spec}\,K^{\circ}[T_{0},\dots,T_{n},T^{-1}_{0},\dots,T^{-1}_{m},\varpi^{-1}]/(T_{0}\cdots T_{r}-\varpi)\\&\cong \textsf{Spec}\,K[T_{0},\dots,T_{n},T^{-1}_{0},\dots,T^{-1}_{m}]/(T_{0}\cdots T_{r}-\varpi)
\end{align*}
thus $D_{\mathcal{X}} \defeq \mathcal{X} \setminus \mathcal{U}_{\mathcal{X}} \subset \mathcal{X}_{K}$ and it is a normal crossing divisor by $D \defeq \mathcal{M} \setminus \mathcal{U} = V(T_{0} \cdots T_{m})$. It's easy to see when $m=r$,  we have $D=\mathcal{X}_{\widetilde{K}}$. Now we take the completion along special fibers for \ref{standard semi-stable model}, \'etale locally, we have an \'etale morphism of $K^{\circ}$-formal schemes $\widehat{g}:(\widehat{\mathcal{X}},\widehat{\mathcal{U}_{\mathcal{X}}}) \longrightarrow (\widehat{\mathcal{M}},\widehat{\mathcal{U}})$, then take the Berthelot functor on $\widehat{g}$, we have: $$\widehat{g}_{\eta}:(\widehat{\mathcal{X}}_{\eta},\widehat{\mathcal{U}_{\mathcal{X}}}_{\eta}) \longrightarrow (\mathcal{M} \Big(K \llangle T_{0},\dots,T_{n}\rrangle / (T_{0}\cdots T_{r}-\varpi)\Big), \widehat{\mathcal{U}}_{\eta})$$
Let $\widehat{\mathcal{M}}_{\eta} \defeq \mathcal{G}(n,r,m)$, then we have $\mathcal{G}(n,r,m) \cong \mathcal{G}(r,r,r) \times_{K} (\widehat{\A}^{n-r}_{K^{\circ}})_{\eta}$, for $\mathcal{G}(r,r,r)$
\begin{equation}\label{mono-chart-semistable}
    \varphi_{\mathcal{M}}: \mathcal{G}(r,r,r) \hookrightarrow \mathbb{G}^{r,\textsf{an}}_{{m},K} 
\end{equation}
such that $\mathcal{G}(r,r,r)$ is an affinoid domain of $\mathbb{G}^{r,\textsf{an}}_{{m},K}$ and for $(\widehat{\A}^{n-r}_{K^{\circ}})_{\eta}$ has an open covering consists of formal tori. From the description above, we can summarize the strictly semi-stable as Zariski locally, there is a smooth morphism $$\mathcal{X} \longrightarrow \textsf{Spec}\,K^{\circ}[T_{0},\dots,T_{n}]/(T_{0}\cdots T_{n}-\varpi).$$
\begin{remark}
  When $K^{\circ}$ is discrete-valued, then any strictly semi-stable (formal) scheme is a nondegenerate polystable (formal) scheme \cite[3.1]{MR2263704}.   
\end{remark}

\subsection{The Complement of Skeleton for Curves}\label{complement of sk for curve}
\par

In this subsection, we briefly review the structure of the skeleton of curves studied in \cite{MR2395137} and \cite{MR3204269}. For a smooth and connected analytic $K$-curve $X$, a \textit{triangulation} of $X$ is a discrete and closed subset $\mathbb{V}$ consists of rigid or type-2 or type-3 points such that $X \setminus \mathbb{V}$ gives connected component decomposition as a disjoint union of virtual open disks and a finite disjoint union of virtual open annuli. When the base field $K$ is algebraically closed, we have the connected component decomposition as:
\begin{equation*}
    X \setminus \mathbb{V} \cong \coprod \mathbb{D}(0,1)_{K} \smallcoprod \coprod^{n}_{i=1} \mathbb{B}(0,a_{i},1)_{K}
\end{equation*}
The skeleton of $X$ associated to $\mathbb{V}$
is the set:
\begin{equation*}
    \textsf{Sk}(X,\mathbb{V}) \defeq \mathbb{V} \smallcoprod \coprod^{n}_{i=1} \textsf{Sk}(\mathbb{B}(0,a_{i},1)_{K})
\end{equation*}
\begin{situation}
    Let $\mathfrak{C}$ be a semi-stable formal curve over $K^{\circ}$, consider the reduction map $\textsf{red}_{\mathfrak{C}}: \mathfrak{C}_{\eta} \longrightarrow |\mathfrak{C}_{\widetilde{K}}|=|\mathfrak{C}|$, we denote $\mathfrak{C}^{0}_{\widetilde{K}}$ be the set of generic points of the irreducible components of the special fiber $\mathfrak{C}_{\widetilde{K}}$, then the preimage $\textsf{red}^{-1}(\mathfrak{C}^{0}_{\widetilde{K}})$ gives a triangulation $\mathbb{V}_{\mathfrak{C}}$ of $\mathfrak{C}_{\eta}$. In this case, we have:
    \begin{proposition}\cite[Proposition 4.10.]{MR3204269}
    The Berkovich skeleton $\textsf{Sk}({\mathfrak{C}}_{\eta})$ is naturally identified with the skeleton $\textsf{Sk}({\mathfrak{C}}_{\eta},\mathbb{V}_{{\mathfrak{C}}})$.
\end{proposition}

\begin{lemma}\cite[Lemma 3.4.]{MR3204269}
   The connected components of $\mathfrak{C}_{\eta}\setminus \textsf{Sk}({\mathfrak{C}}_{\eta}) \defeq \textsf{Sk}({\mathfrak{C}}_{\eta})^{\textsf{c}}$ are open unit disks. 
\end{lemma}
\begin{remark}
    Essentially, the structure of the complement of the skeleton of a curve is induced by the complement of the skeleton of an open annulus \cite[Lemma 2.12.]{MR3204269}
\end{remark}
For a non-algebraically closed field $K$, consider a $K$-analytic curve $C$ admitting a semi-stable formal model $\mathfrak{C}$, then for the formal curve $\mathfrak{C}_{\widehat{K^{\textsf{alg}}}^{\circ}}$, and the morphism $\pi: \mathfrak{C}_{\widehat{K^{\textsf{alg}}}^{\circ}} \longrightarrow \mathfrak{C}$, we have the following lemma:
\begin{lemma}
For generic fiber $\pi_{\eta}: C_{\widehat{K^{\textsf{alg}}}} \longrightarrow C$, then $\pi_{\eta}^{-1}(\textsf{Sk}(\mathfrak{C}_{\eta}))=\textsf{Sk}({\mathfrak{C}_{\eta}}_{\widehat{K^{\textsf{alg}}}^{\circ}})$.    
\end{lemma}
\begin{remark}\label{Complement decomposition weak form}
    Note that $\pi_{\eta}$ is surjective, we have $\pi_{\eta}(\textsf{Sk}({\mathfrak{C}_{\eta}}_{\widehat{K^{\textsf{alg}}}^{\circ}})^{\textsf{c}})=\textsf{Sk}(\mathfrak{C}_{\eta})^{\textsf{c}}$, thus $\textsf{Sk}(\mathfrak{C}_{\eta})^{\textsf{c}}=\bigcup \pi_{\eta}(\mathbb{D}(0,1)_{\widehat{K^{\textsf{alg}}}})=\bigcup \mathbb{D}_{K}$, but this not a connected component decomposition. 
\end{remark}
\end{situation}

\subsection{The complement of skeleton for higher dimensional $K$-analytic space}
In this subsection, we will show that for a $K$-analytic space of dimension greater than 1 which admits a strictly semi-stable formal model, for every point of the complement of its Berkovich skeleton we may find a relative dimension 1 fibration such that this point is not contained in the skeleton of an analytic curve which is contained inside a fiber. Let's recall that a monomial chart of a $K$-analytic space is a morphism $f: U \longrightarrow \mathbb{G}^{n,\textsf{an}}_{m,K}$ such that $U$ is an analytic domain of $X$ and $f$ has zero dimensional fibers. The key observation is the following result:
\begin{lemma}\label{semistable fiber}
    With the notation in \ref{standard semi-stable model}, for the projections $\textsf{fp}_{i}: \mathbb{G}_{m,K}^{n,\textsf{an}} \longrightarrow \mathbb{G}_{m,K}^{n-1,\textsf{an}} $, where $1 \leq i \leq n$, by forgetting the $i$-th component, consider the following compositions:  $\varphi_{i} \defeq \textsf{fp}_{i} \circ \varphi_{M}$ and $g_{i} \defeq \varphi_{i} \circ \widehat{g}_{\eta}$, then the induced morphisms on the fibers have :
    \begin{enumerate}
        \item $\varphi_{M,x,i}: \varphi_{i}^{-1}(x) \longrightarrow \mathbb{G}_{m,\mathscr{H}(x)}^{1,\textsf{an}}$ is an analytic locally monomial chart such that $\varphi^{-1}_{M,x,i}(\textsf{Sk}(\mathbb{G}_{m,\mathscr{H}(x)}^{1,\textsf{an}}))$ consists all monomial points of $\varphi_{i}^{-1}(x)$ for any $x \in \mathbb{G}_{m, K}^{n-1,\textsf{an}} $.
        \item Analytic locally, $\widehat{g}_{\eta}: g^{-1}_{i}(x) \longrightarrow \varphi^{-1}_{i}(x)$ is a generic fiber of \'etale morphism of strictly semi-stable formal schemes. 
    \end{enumerate}
    \begin{proof}
        \begin{enumerate}
            \item Without loss of generality, let's assume $i=n$, note that 
            \begin{equation*}
                M=\{x \in \mathbb{G}^{n, \textsf{an}}_{m,K}\,|\,|T_{1} \cdots T_{n}|_{x} \geq |\varpi| ; |T_{i}|_{x} \leq 1 \,\,\text{for}\,\, 1 \leq i \leq n \}
            \end{equation*}
 we have 
            \begin{align*}
                \textsf{fp}_{n}(M)&=\bigcup_{t}\{x \in \mathbb{G}^{n-1, \textsf{an}}_{m,K}\,|\,|T_{1} \cdots T_{n-1}|_{x} \geq \frac{|\varpi|}{t} ; |T_{i}|_{x} \leq 1\,\,\text{for}\,\, 1 \leq i \leq n-1 \} 
            \end{align*} where $t=|T_{n}|_{x}$ for $x \in M $.
            Note that $\frac{|\varpi|}{t}$ does not necessarily belong to the divisible value group $\sqrt{|K^{\times}|}$. Still, we can always take the union of strictly affinoid subdomains in the form of $\mathcal{U} \defeq \bigcup_{\alpha} \mathcal{U}_{\alpha} = \bigcup_{\alpha}\{x \in \mathbb{G}^{n-1, \textsf{an}}_{m,K}\,|\,|T_{1} \cdots T_{n-1}|_{x} \geq |\varpi_{\alpha}| ; |T_{i}|_{x} \leq 1\,\,\text{for}\,\, 1 \leq i \leq n-1 \} $ to cover $\textsf{fp}_{n}(M)$ such that $\varpi_{\alpha} \in K^{\times}$ and $|\varpi_{\alpha}| \leq 1$. Meanwhile, we have $M=\bigcup_{\alpha} \textsf{fp}^{-1}_{n}(\mathcal{U}_{\alpha}) \cap M =\bigcup_{\alpha}\{x \in \mathbb{G}^{n, \textsf{an}}_{m,K}\,|\,|T_{1} \cdots T_{n-1}|_{x} \geq |\varpi_{\alpha}| ; |T_{i}|_{x} \leq 1\,\,\text{for}\,\, 1 \leq i \leq n; |T_n|_{x} \geq |\varpi  \varpi_{\alpha}^{-1}| \}$. Thus
              \begin{align*}
                \varphi_{n}^{-1} (x) & \cong \mathcal{M}\Big(\frac{K\llangle X_{0}, \dots, X_{n} \rrangle}{(X_{0}\cdots X_{n}-\varpi)}\Big) \times_{\mathbb{G}_{m,K}^{n-1,\textsf{an}}}  \mathcal{M}(\mathscr{H}(x))  \\  & \cong \mathcal{M}\Big(\frac{K\llangle X_{0}, \dots, X_{n} \rrangle}{(X_{0}\cdots X_{n}-\varpi)}\Big) \times_{\mathcal{U}}  \mathcal{M}(\mathscr{H}(x))  \\ & \cong \bigcup_{\alpha} (\textsf{fp}^{-1}_{n}(\mathcal{U}_{\alpha}) \cap M ) \times_{\mathcal{U}_{\alpha}} \mathcal{M}\big(\mathscr{H}(x)\big)\\ & \cong \bigcup_{\alpha} \mathcal{M}\Big(\frac{K\llangle X_{0},\dots,X_{n-1}, X_{n}, \varpi  \varpi^{-1}_{\alpha} X^{-1}_{n} \rrangle}{(X_{0}\cdots X_{n-1} - \varpi_{\alpha})}\Big) \times_{\mathcal{U}_{\alpha}} \mathcal{M}\big(\mathscr{H}(x)\big)
                \\ &  \cong \bigcup_{\alpha} \mathcal{M} \big(\mathscr{H}(x)\llangle T, \varpi  \varpi^{-1}_{\alpha} T^{-1 }\rrangle\big) \\ & \cong \bigcup_{\alpha} \mathbb{B}(0; |\varpi\varpi^{-1}_{\alpha}|,1)_{\mathscr{H}(x)}
            \end{align*} 
            Note the morphism $\varphi_{M,x,n}$ is flat. Moreover, by \cite[2.3]{MR3204269}

            \begin{align*}
                \varphi^{-1}_{M,x,n}(\textsf{Sk}(\mathbb{G}_{m,\mathscr{H}(x)}^{1,\textsf{an}})) & =\bigcup_{\alpha}\textsf{Sk}\Big(\big(\textsf{Spf}\,\frac{\mathscr{H}(X)^{\circ}\llangle X_{0},X_{1} \rrangle}{(X_{0}X_{1}-\varpi  \varpi_{\alpha}^{-1})}\big)_{\eta}\Big) \\ & =
                \bigcup_{\alpha} \textsf{Sk}\big(\mathbb{B}(0; |\varpi\varpi^{-1}_{\alpha}|,1)_{\mathscr{H}(x)}\big) \\ & =  \varphi^{-1}_{n}(x)^{\textsf{mon}}
            \end{align*}
            where the last equality comes from \cite[7.2.7.]{MR3702313}

            \item Since $g_{n}^{-1}(x) \cong \widehat{\mathcal{V}}_{\eta} \times_{\mathbb{G}^{n-1,\textsf{an}}_{m,K}} \mathcal{M}\big(\mathscr{H}(x)\big) \cong \bigcup_{\alpha} \big(\widehat{\mathcal{V}}_{\eta} \times_{M} ({\textsf{fp}^{-1}_{n}(\mathcal{U}_{\alpha}}) \cap M )\big)\times_{\mathcal{U}_{\alpha}} \mathcal{M}\big(\mathscr{H}(x)\big) $, now we take:
            \begin{equation}
               \begin{tikzcd}[cramped, sep=scriptsize]
               \mathfrak{g}_{n,\alpha ,x}: \mathfrak{G}_{n,\alpha ,x} \defeq \big(\widehat{\mathcal{V}} \times_{\widehat{\mathcal{M}}} \textsf{Spf}\frac{K^{\circ}\llangle X_{0},\dots,X_{n-1}, X_{n}, \varpi  \varpi^{-1}_{\alpha} X^{-1}_{n} \rrangle}{(X_{0}\cdots X_{n-1} - \varpi_{\alpha})}\big) \times_{\textsf{Spf}\,\frac{\mathscr{H}(X)^{\circ}\llangle X_{0},\dots,X_{n-1} \rrangle}{(X_{0} \cdots X_{n-1}-\varpi_{\alpha})}} \textsf{Spf}\,\mathscr{H} (x)^{\circ} \arrow[d,"\textsf{\'et}"] & \\ \textsf{Spf}\,\frac{\mathscr{H}(X)^{\circ}\llangle X_{0},X_{1} \rrangle}{(X_{0}X_{1}-\varpi  \varpi_{\alpha}^{-1})}
               \end{tikzcd}
            \end{equation}
            Then $\mathfrak{g}_{n,\alpha ,x}$ is \'etale and $({\widehat{g}_{\eta})_{x}}|_{g^{-1}_{n}(\mathcal{U}_{\alpha})}=({\mathfrak{g}_{n,\alpha ,x}})_{\eta}$, thus $\mathfrak{G}_{n,\alpha ,x}$ is strictly semi-stable. In this case, we have $$\textsf{Sk}(({\mathfrak{G}}_{n,\alpha ,x})_{\eta})=(\mathfrak{g}_{n,\alpha ,x})_{\eta}^{-1}(\textsf{Sk}\big(\mathbb{B}(0; |\varpi\varpi^{-1}_{\alpha}|,1)_{\mathscr{H}(x)}\big))$$ and  $g^{-1}_{n}(x)^{\textsf{mon}}=\bigcup_{\alpha} \textsf{Sk}((\mathfrak{G}_{n,\alpha ,x})_{\eta})$ by \cite[Theorem 5.2.]{MR1702143}

        \end{enumerate}
    \end{proof}
\end{lemma}

 \subsection{The Complement of Skeleton}
\begin{situation}
     We start with the torus model $\mathbb{G}^{n,\textsf{an}}_{{m},K}=\textsf{Spec}^{\textsf{an}}\,K[t_{1}^{\pm},t_{2}^{\pm},\dots,t_{n}^{\pm}] $, we have $\textsf{Sk}(\mathbb{G}^{n,\textsf{an}}_{{m},K})$ is homeomorphic to $\bigtimes^{n}_{i=1} \textsf{Sk}(\mathbb{G}^{1,\textsf{an}}_{{m},K})$, where the skeleton of analytic torus $\mathbb{G}^{n,\textsf{an}}_{m,K}$ consists of generalized Gauss valuations which are defined as follows:
     Any generalized Gauss valuation $|\!|-|\!|_{x}$ associated with $\underline{r}=(r_{1},r_{2},\dots,r_{n}) \in \mathbb{R}^{n}_{>0}$ on $K[\underline{t}]$ is given by the formula:
     \begin{equation*}
         |\!|\sum_{i \in \mathbb{N}^{n}}a_{i}\underline{t}^{i}|\!|_{x} = \textsf{max}_{i}\{|a_{i}|\prod^{n}_{j=1}r^{i_{j}}_{j}\}
     \end{equation*}
     Then we have $\textsf{Sk}(\mathbb{G}^{n,\textsf{an}}_{{m},K})^{\textsf{c}}=\mathbb{G}^{n,\textsf{an}}_{{m},K}\setminus \Big(\bigtimes^{n}_{i=1} \textsf{Sk}(\mathbb{G}^{1,\textsf{an}}_{{m},K})\Big) $.
     Now we consider the projection maps $\textsf{fp}_{i}: \mathbb{G}^{n,\textsf{an}}_{{m},K} \longrightarrow \mathbb{G}^{n-1,\textsf{an}}_{{m},K} $ for $1 \leq i \leq n$,  and we have $\textsf{fp}_{i}: \textsf{Sk}(\mathbb{G}^{n,\textsf{an}}_{{m},K}) \longrightarrow \textsf{Sk}(\mathbb{G}^{n-1,\textsf{an}}_{{m},K})$. Now we show the following lemma:
     \begin{lemma}\label{the complement of skeleton of torus}
         The complement of skeleton $\textsf{Sk}(\mathbb{G}^{n,\textsf{an}}_{{m},K})^{\textsf{c}}$  is a union of open virtual disks.
     \end{lemma}
     \begin{proof}
            Let's begin by writing $\G_{m,K}^{n,\textsf
            {an}} \cong \G_{m,K}^{n-1,\textsf
            {an}} \times_{K} \G_{m,K}^{1,\textsf
            {an}}$. We denote $\textsf{pr}_{1} \defeq \textsf{fp}_{n}$ and $\textsf{pr}_{2}$ as the projections from $\G_{m,K}^{n,\textsf{an}}$ to $\G_{m,K}^{n-1,\textsf{an}}$ and $\G_{m,K}^{1,\textsf{an}}$ respectively. When $n=1$, this is a result of \cite[Lemma 2.12.]{MR3204269}. We proceed by induction on $n$. Assuming the complement of skeleton $\textsf{Sk}(\mathbb{G}^{n-1,\textsf{an}}_{{m}, K})^{\textsf{c}}$ is a union of open virtual disks, we claim that that a point belongs to $\textsf{Sk}(\mathbb{G}^{n,\textsf{an}}_{m,K})$ if and olny if its projection belongs to $\textsf{Sk}(\mathbb{G}^{n-1,\textsf{an}}_{m,K})$ and the point belongs to the skeleton of its fiber.  
Let $x \in \textsf{Sk}(\mathbb{G}^{n,\textsf{an}}_{{m}, K})^{\textsf{c}}$, then we will show our claim by testing whether $\{\textsf{pr}_{i}(x)\}_{i}$ belongs to the skeletons of the images of $\{\textsf{pr}_{i}\}_{i}$ or not. Without loss of generality, we only discuss the following case: If $\textsf{pr}_{1}(x) \in \textsf{Sk}(\mathbb{G}^{n-1,\textsf{an}}_{{m},K})$ and $\textsf{pr}_{2}(x) \in \textsf{Sk}(\mathbb{G}^{1,\textsf{an}}_{{m},K}) $, then the fiber $\textsf{pr}^{-1}_{2}(\textsf{pr}_{2}(x)) \cong \mathbb{G}^{n-1,\textsf{an}}_{{m},\mathscr{H}(\textsf{pr}_{2}(x)) }$ we claim that $x \notin \textsf{Sk}(\textsf{pr}^{-1}_{2}(\textsf{pr}_{2}(x)))$. This can be seen by noting that $\textsf{Sk}(\textsf{pr}^{-1}_{2}(\textsf{pr}_{2}(x)))=\textsf{Sk}(\mathbb{G}^{n-1,\textsf{an}}_{{m},K}) \times \mathcal{M}(\mathscr{H}(\textsf{pr}_{2}(x)))$, if $x \in \textsf{Sk}(\textsf{pr}^{-1}_{2}(\textsf{pr}_{2}(x)))$, let $f=\sum_{i \in \mathbb{N}^{n}}a_{i}\underline{t}^{i}=\sum_{j \in \mathbb{N}^{n-1}}b_{j}(t_{n})\underline{t}^{j} \in K[\underline{t}]=K[t_{n}][t_{1}\cdots t_{n-1}]$ where $i=(i_{1},i_{2},\dots,i_{n})$, $j=(j_{1},j_{2},\dots,j_{n-1})$, and $b_{j}(t_{n})=\sum_{s}a_{js}t^{s}_{n} \in K[t_{n}]$. Then we have $$|\!|f|\!|_{x}=\textsf{max}_{j}\{|\!|b_{j}(t_{n})|\!|_{\textsf{pr}_{2}(x)}\underline{r}^{j}\}=\textsf{max}_{j}\{|\!|b_{j}(t_{n})|\!|_{\textsf{pr}_{2}(x)}\prod^{n-1}_{t=1}r^{j_{t}}_{t}\}$$
Meanwhile, we have $$|\!|b_{j}(t_{n})|\!|_{\textsf{pr}_{2}}=\textsf{max}_{s}\{|a_{js}|r^{s}_{n}\}$$ Thus we have:
$$|\!|f|\!|_{x}=\textsf{max}_{i}\{|a_{i}|(\prod^{n-1}_{t=1}r^{j_{t}}_{t})\cdot r^{s}_{n}\}$$ which is a generalized Gauss valuation, contradicts to our assumption. 
Then there is an open virtual disk $\mathbb{D}_{L}$ such that $x \in \mathbb{D}_{L} \subseteq \textsf{Sk}(\mathbb{G}^{n-1,\textsf{an}}_{{m},\mathscr{H}(\textsf{pr}_{2}(x)) })^{\textsf{c}} $ with field extension $L|\mathscr{H}(\textsf{pr}_{2}(x))$. Meanwhile, assume there exists $x_{0} \in \mathbb{D}_{L}  \cap \textsf{Sk}(\mathbb{G}^{n,\textsf{an}}_{{m}, K})$, then $x_{0} \in \textsf{Sk}(\mathbb{G}^{n-1,\textsf{an}}_{{m},\mathscr{H}(\textsf{pr}_{2}(x)) })$, which leads to a contradiction. Thus we have $\textsf{Sk}(\mathbb{G}^{n,\textsf{an}}_{{m},K})^{\textsf{c}}= \bigcup \mathbb{D}_{L}$ with complete non-Archimedean field extension $L|K$. In fact, from the induction procedure, we can see that each disk $\mathbb{D}_{L}$ comes from a certain fiber $\textsf{fp}^{-1}_{i}(\textsf{fp}_{i}(x))$.

     \end{proof} 
\end{situation}

\begin{lemma}\label{The decomposition of complement of skeleton r=n}
    For a strictly semi-stable scheme $\mathcal{X}$ such that $r=n$ in \ref{standard semi-stable model}, the complement of Berkovich skeleton $\textsf{Sk}({\widehat{\mathcal{X}}}_{\eta})^{\textsf{c}}$ is a union of open virtual disks.
    \begin{proof}
        It suffices to study $\textsf{Sk}({\widehat{\mathcal{V}}}_{\eta})^{\textsf{c}}$, then for any $x \in \textsf{Sk}({\widehat{\mathcal{V}}}_{\eta})^{\textsf{c}}$, consider $x' \defeq \varphi_{M} \circ (\widehat{g})_{\eta}(x) \in \textsf{Sk}(\mathbb{G}^{n,\textsf{an}}_{m,K})^{\textsf{c}}$, by  Lemma \ref{the complement of skeleton of torus}, there exists $i$ for $1 \leq i \leq n $ such that $x' \in \textsf{Sk}(\textsf{fp}_{i}^{-1}(\textsf{fp}_{i}(x'))^{\textsf{c}}$, then by Lemma  \ref{semistable fiber}, $x \notin g^{-1}_{i}(\textsf{fp}_{i}(x'))^{\textsf{mon}}$, thus there exists an analytic curve $({\mathfrak{G}}_{n,\alpha ,x})_{\eta}$ such that $x \in \textsf{Sk}(({\mathfrak{G}}_{n,\alpha ,x})_{\eta})^{\textsf{c}}$.
    \end{proof}
\end{lemma}
\begin{remark}\label{good reduction case}
    In fact, Lemma \ref{The decomposition of complement of skeleton r=n} still holds for $$\mathcal{M} \defeq \textsf{Spec}\,K^{\circ}{[T_{0},\dots,T_{n}]}/(T_{0} \cdots T_{n}-a)$$ with $|a|=1$. In this case, $\widehat{\mathcal{M}} \cong \textsf{Spf}(K^{\circ}\llangle T^{\pm 1}_{1},\dots T_{n}^{\pm 1}\rrangle)$ and $\widehat{\mathcal{M}}_{\eta} \cong \big(\mathbb{B}(0; 1, 1)_{K}\big)^{n}$ which is a product of generalized annulus and contractible. Consider a fibration similar to we did in Lemma \ref{semistable fiber}, then we get a family of  \'etale morphisms $\{{(\widehat{g}_{\eta})_{L}}: {(\widehat{\mathcal{V}}_{\eta})_{L}} \longrightarrow \mathbb{B}(0; 1, 1)_{L}\}_{L}$ for field extension $L|K$, and we have $(\widehat{g}_{\eta})^{-1}_{L}(\textsf{Sk}(\mathbb{B}(0;1,1))_{L})=\textsf{Sk}(({\widehat{\mathcal{V}}_{\eta})_{L}})$ by \cite[Lemme 2.2.14.]{thuillier:tel-00010990}. The rest part are the same, and we are done.
\end{remark}
\begin{theorem}\label{main theorem 1}
     Let ${{\mathcal{X}}}$ be a strictly semi-stable scheme over complete discrete valuation ring $K^{\circ}$ defined in \ref{standard semi-stable model}, the complement of Berkovich skeleton $\textsf{Sk}({\widehat{\mathcal{X}}}_{\eta})^{\textsf{c}}$ is a union of open virtual disks.
\begin{proof}
Zariski locally, note that $g$ could be factorized as 
\begin{equation*}
    \begin{tikzcd}
        (\mathcal{V},U_{\mathcal{V}}) \arrow[r,"h"] \arrow[rr, bend right=18, "g"] & (\mathbb{A}_{\mathcal{M}}^{s}, U_{\mathbb{A}_{\mathcal{M}}^{n}}) \arrow[r,"\textsf{pr}_{\mathcal{M}}"]  & (\mathcal{M},U_{\mathcal{M}}) 
    \end{tikzcd}
\end{equation*}

  where $h$ is \'etale and $s$ is the relative dimension of $g$. Now let's study the structure of $\textsf{Sk}({\widehat{\mathcal{V}}}_{\eta})^{\textsf{c}}$. We first note that the $K^{\circ}$-scheme ${\mathbb{A}}^{s} \times_{K^{\circ}} {\mathcal{M}}$ is a strictly nondegenerate polystable scheme. More precisely, we have a locally finite covering of open formal subscheme $\{\widehat{\mathcal{W}}\}$ for $\widehat{\mathbb{A}}^{s} \times_{K^{\circ}} \widehat{\mathcal{M}}$ satisfying the following diagram:

\begin{equation}
\begin{tikzcd}
	& {\mathfrak{W}' \defeq \widehat{\mathcal{V}} \times_{\widehat{\mathbb{A}}^{s} \times_{K^{\circ}} \widehat{\mathcal{M}}} \widehat{\mathcal{W}}} \\
	{\widehat{\mathcal{V}}} && {\widehat{\mathcal{W}}} \\
	& {\widehat{\mathbb{A}}^{s} \times_{K^{\circ}} \widehat{\mathcal{M}}} && {\textsf{Spf}(A_{0}) \times \cdots \times \textsf{Spf}(A_{p})}
	\arrow["\textsf{open}", hook', from=1-2, to=2-1]
	\arrow["{\textsf{\'et}}"', from=1-2, to=2-3]
	\arrow["{\textsf{\'et}}"', from=2-1, to=3-2]
	\arrow["\textsf{open}", hook', from=2-3, to=3-2]
	\arrow["{\textsf{\'et}}"', from=2-3, to=3-4]
\end{tikzcd}
\end{equation}
where each $\textsf{Spf}(A_{i})$ is of form $\textsf{Spf}\big(K^{\circ} \llangle T_{0},\dots,T_{n_{i}}\rrangle / (T_{0}\cdots T_{n_{i}}-a)\big)$ for $0 \neq a \in K^{\circ} $. Then through the construction of \cite[4.6.]{MR3479562}, it suffices to study the structure of $\textsf{Sk}({\mathfrak{W}'}_{\eta})^{\textsf{c}}$, but this is a immediate result follows from Lemma \ref{The decomposition of complement of skeleton r=n} and Remark \ref{good reduction case}.

\end{proof}
\end{theorem}

\begin{situation}
It's natural to study the structure of the complement of the Berkovich skeleton for a log regular formal scheme as a generalization. Let $\mathcal{X}$ be a proper log regular scheme over $K$ and $\mathfrak{X}$ be a log regular formal model of $\mathcal{X}$ over $K^{\circ}$. In \cite{MR3963489}, the authors show one can embed the Kato fans $\textsf{F}(\mathfrak{X})$ defined in \cite{MR1296725} to $\mathcal{X}^{\textsf{an}}$ as a generalization of Berkovich skeleton of a strictly semi-stable model defined in \cite{MR1702143}.  The elementary example is a generalized semistable family as follows:

\par
Consider $M \defeq \spec\,K^{\circ}[T_{0},\dots,T_{n}]/(T^{e_{0}}_{0}\cdots T^{e_{n}}_{n}-\varpi)$ such that $(e_{i},p)=1$ for $0 \leq i \leq n$, where $p=\textsf{char}(\widetilde{K})$. Then for the following morphism:
    \begin{align*}
        f: \widetilde{M}  \defeq  \spec\,\frac{K^{\circ}[X_{0},\dots,X_{n}]}{(X^{e}_{0}\cdots X^{e}_{n}-\varpi)}  \longrightarrow  \spec\,\frac{K^{\circ}[T_{0},\dots,T_{n}]}{(T^{e_{0}}_{0}\cdots T^{e_{n}}_{n}-\varpi)}
    \end{align*}
    which is induced by $T_{i} \longmapsto X^{\frac{e}{e_{i}}}_{i}$ we have:
    \begin{proposition}\cite[cf.Theorem 2.2.]{MR2923183}\label{saito transformation}
        $f$ is an \textsf{fppf} covering space such that its generic fiber
             $f_{K}$ is a Galois \'etale covering space. Moreover,
             the preimage of $\textsf{Sk}(\widehat{M}_{\eta})$ of $\widehat{f}_{\eta}$ is $\textsf{Sk}((\widehat{\widetilde{M}})_{\eta})$.
        
        \begin{proof}
            \begin{enumerate}
                \item The surjectivity of $f$ is easy to verify on the closed points of restriction of generic fiber and special fiber.
                \item The \'etaleness of $f_{K}$ is standard and $\textsf{Aut}(\widetilde{M}_{K}|M_{K}) \cong \prod_{i} \boldsymbol{\mu}_{\frac{e}{e_{i}}}$.
                \item By the construction of the Berkovich skeleton, we have for any $x \in \textsf{Sk}(\widehat{M}_{\eta})$, we have $|\sum_{I}a_{I}\underline{T}_{I}^{m_{I}}|_{x}=\textsf{max}_{I}|a_{I}|\textsf{exp}(-\sum^{n}_{i=0}m_{i}\cdot v_{i})$ for a fixed $v_{I}=(v_{i})_{0 \leq i \leq n} \in \mathbb{R}_{>0}^{n+1}$ satisfies $\sum_{i}e_{i}\cdot v_{i}=1$ and for $x' \in \textsf{Sk}((\widehat{\widetilde{M}})_{\eta})$, we have $|\sum_{I}a'_{I}\underline{X}_{I}^{r_{I}}|_{x'}=\textsf{max}_{I}|a'_{I}|\textsf{exp}(-\sum^{n}_{i=0}r_{i}\cdot v'_{i})$ for a fixed $v'_{I}=(v'_{i})_{0 \leq i \leq n} \in \mathbb{R}_{>0}^{n+1}$ satisfies $\sum_{i}e \cdot v'_{i}=1$. Then the result is obvious by direct calculation. 
            \end{enumerate}
        \end{proof}
    \end{proposition}
\begin{situation}
    Let $G \defeq \prod_{i} \boldsymbol{\mu}_{\frac{e}{e_{i}}}$ as an \'etale abelian sheaf, then we study the $G$-action on $(\widehat{\widetilde{\mathcal{M}}})_{\eta}$ The complement of the skeleton $\textsf{Sk}((\widehat{\widetilde{M}})_{\eta})^{\textsf{c}}$ is $G$-stable. To see this, firstly by \ref{saito transformation}, we have the generic fiber $f_{K}$ of $f$ is a Galois-\'etale covering space and $\mathcal{M}_{K} \cong \widetilde{\mathcal{M}}_{K}/G$. After taking the analytification and Berthelot generic fiber for $f$ and $f_{K}$, we have the following diagram: 

    \begin{equation}\label{generalized semistable base change}
    \begin{tikzcd}[row sep=scriptsize, column sep=scriptsize]
       (\widehat{\widetilde{\mathcal{M}}})_{\eta}  \arrow[r,"\widehat{f}_{\eta}"] \arrow[d,hookrightarrow] & \widehat{\mathcal{M}}_{\eta} \arrow[d,hookrightarrow] & \\
            \widetilde{\mathcal{M}}^{\textsf{an}}_{K} \arrow[r,"f^{\textsf{an}}_{K}"] & \mathcal{M}^{\textsf{an}}_{K}
    \end{tikzcd}
\end{equation}

Now, by \cite[Corollaire 1.1.50.]{MR3381140} and \cite[(0.3.5) Proposition]{berthelot1996cohomologie}, we have $\widehat{f}_{\eta}$ is  is finite, \'etale and surjective.
Thus we have $\widehat{\mathcal{M}}_{\eta} \cong \widehat{\widetilde{\mathcal{M}}}_{\eta}/G$ as the orbit space by \cite[Theorem 1.3.6.]{MR4705884} Furthermore, consider the restriction of $\widehat{f}_{\eta}$ on $\textsf{Sk}(\widehat{\widetilde{\mathcal{M}}}_{\eta})^{\textsf{c}}$, we have $\textsf{Sk}(\widehat{\widetilde{\mathcal{M}}}_{\eta})^{\textsf{c}}/G \cong \textsf{Sk}(\widehat{\mathcal{M}}_{\eta})^{\textsf{c}}$, thus, the complement of skeleton $\textsf{Sk}(\widehat{\mathcal{M}}_{\eta})^{\textsf{c}}$ is the $G$-orbit space of the complement of skeleton $\textsf{Sk}(\widehat{\widetilde{\mathcal{M}}}_{\eta})^{\textsf{c}}$.  Unfortunately, we still don't know how the disks decomposition behaves under the group action $G$, even for the curve case.  

\end{situation}
\end{situation}

\section{Characterization of essential skeleton}
In this section, we assume any $K$-analytic spaces to be smooth of pure dimension $n$. Then we give a criterion for points belonging to the essential skeleton of a proper smooth variety in terms of open unit disks when the ground field $K$ is a fixed non-Archimedean field with equal characteristic $0$.
\begin{situation}\textit{K\"ahler Seminorm on Differentials}.
We begin recall some basic definitions and facts about K\"ahler seminorm on the module of differentials of a seminormed ring map.

\begin{definition}
    A \textit{seminormed group} is an abelian group $R$ equipped with a map
\begin{equation*}
    | \cdot |_R : R \to \mathbb{R}_{\geq 0}
\end{equation*}
satisfying the following properties:

\begin{itemize}
    \item One has $|0|_R = 0$. 
    \item For all $x, y \in R$, one has $|x - y|_R \leq \textsf{max}\big( |x|_R, |y|_R\big)$.
\end{itemize}
A \textit{seminorm} on a ring $R$ is a seminorm $| \cdot |_{R}$ on the underlying group $(R, +)$ that satisfies $|1|_{R} = 1$ and $|xy|_{R} \leq |x|_{R}|y|_{R}$ for any $x, y \in R$. A ring with a fixed seminorm is called a \textit{seminormed ring}. \textit{A seminormed $R$-module $M$} is an $R$-module provided  
with a seminorm $|\!| \cdot |\!|_{M}$ such that  
\[
|\!| r m |\!|_M \leq |r|_R |\!| m |\!|_M
\]
for any \( r \in R \) and \( m \in M \).

A \textit{map of seminormed groups (resp. rings)} $(R, |\cdot|_R) \to (S, |\cdot|_S)$ is a map of groups (resp. rings) $f : R \to S$ such that
\[
|f(x)|_S \leq |x|_R \quad \text{for all } x \in R.
\]
 The map of seminormed modules is defined in the obvious way.
\end{definition}
\begin{definition}\label{Kahler seminorm}
    Let $R \longrightarrow S$ be a map of seminormed rings, the \textit{K\"ahler seminorm} $|\!|\cdot|\!|_{\Omega}$ on the module of differential $\Omega_{S/R}$ is defined as:
    \begin{equation}
        |\!|x|\!|_{\Omega} \defeq \textsf{inf}_{x = \sum c_{i} ds_{i}}\textsf{max}_{i}|c_{i}|_{S}|s_{i}|_{S}
    \end{equation}
for $x \in \Omega_{S/R}$. The infimum is over all representations of $x$ as a finite sum $x = \sum c_{i} ds_{i}$ and $|\!|x|\!|_{\Omega}$ is independent of the seminorm of $R$. We denote the completion of $(\Omega_{S/R},|\!|\cdot|\!|_{\Omega})$ by $(\widehat{\Omega}_{S/R}, |\!|\cdot|\!|_{\widehat{\Omega}})$, which is a Banach $S$-module and its norm is called \textit{K\"ahler norm}.
\end{definition}
\begin{remark}
    In fact, for a map of Banach rings $R \longrightarrow S$, we have $\widehat{\Omega}_{S/R} \cong \widehat{J/J^{2}}$, where $J = \textsf{Ker}(S \longrightarrow S \widehat{\otimes}_{R}S)$. In the  $K$-affinoid algeras case, we have $\widehat{\Omega}_{S/R} \cong {J/J^{2}}$, thus the definition \ref{Kahler seminorm} coincides with  Berkovich's setting in \cite[3.3.]{MR1259429}
\end{remark}
\end{situation}
\subsection{K\"ahler Skeleton}
In this subsection, we first review the basics of K\"ahler seminorm defined in \cite{MR3702313} and prepare some necessary lemmas we need later.

\begin{situation}\textit{K\"ahler Seminorm}.
\par
Let's consider a morphism of $K$-analytic spaces $X \longrightarrow Y$, then the \textit{K\"ahler seminorm} $|\!|\cdot|\!|_{\Omega,X/Y}$ of $\Omega_{X_{G}/Y_{G}}$ is defined as follows: for any affinoid subdomain $V=\mathcal{M}(B)$ in $X$ such that $f(V) \subseteq U=\mathcal{M}(A)$, $|\!|\cdot|\!|_{\Omega,X/Y}$ of $\Omega_{X_{G}/Y_{G}}$ is the K\"ahler seminorm extended from the K\"ahler norm on $\widehat{\Omega}_{B/A}$.
\par
Consider $X \longrightarrow S$ be a quasi-smooth morphism of $K$-analytic spaces with relative dimension $n$, and the relative canonical sheaf $\mathscr{K}_{{X}/{S}} \defeq \bigwedge^{n}\Omega_{{X}/{S}}$ and we have corresponding seminorm $\lvert\lvert-\lvert\lvert_{\Omega}$ and $\lvert\lvert-\lvert\lvert_{\mathscr{K}}$. Now for a given non-zero global pluricanonical form $\omega 
 \in \Gamma(X, \mathscr{K}_{X/S})$, let $\textsf{M}_{\omega}$ be the maximality of locus. 

    \begin{lemma}\label{etale-sk}
        Let $f:X \longrightarrow Y$ be an \'etale morphism of quasi-smooth $K$-analytic spaces, then $f^{-1}(\textsf{M}_{\omega}) \subseteq \textsf{M}_{f^{*}\omega}$ for any $ \omega \in \Gamma(Y,\mathscr{K}_{Y})$. 
        \begin{proof}
             Note that $f$ is \'etale, we have $f^{*}(\mathscr{K}_{Y}) \cong \mathscr{K}_{X}$, thus $\widehat{\mathscr{K}}_{Y,\mathscr{H}(f(x))|K}$ is isometric to $\widehat{\mathscr{K}}_{X,\mathscr{H}(x)|K}$ for any $x \in X$, then we have  $$\lvert\lvert \omega \lvert\lvert_{f(x_{0})}=\lvert\lvert \omega \lvert\lvert_{\widehat{\mathscr{K}},\mathscr{H}(f(x_{0}))|K}=\lvert\lvert f^{*}\omega \lvert\lvert_{\widehat{\mathscr{K}},\mathscr{H}(x_{0})|K} \geq  \lvert\lvert f^{*}\omega \lvert\lvert_{\widehat{\mathscr{K}},\mathscr{H}(x)|K}$$ for $f(x_{0}) \in \textsf{M}_{\omega}$.
        \end{proof}
    \end{lemma}
\end{situation}
Consider a quasi-smooth $K$-analytic space $X$ with dimension $n$, then for the canonical morphism $g: \overline{X} \defeq X_{\widehat{K^{\textsf{alg}}}} \longrightarrow X$, Then we have the following result when $K$ is equalcharacteristic $0$.
\begin{lemma}\label{pullback of Kahler seminorm}
    The pull back of K\"ahler seminorm on $\Omega_{X}$ along $g$ coincides with K\"ahler seminorm on $\Omega_{\overline{X}}$.
    \begin{proof}
       Let $X' \defeq X \otimes_{K}K'$ where $K'|K$ be any finite extension of $K$, then by \cite[Corollary 6.3.13.]{MR3702313}, $|\!|-|\!|_{\Omega, X'/K'}$ is the pullback of $|\!|-|\!|_{\Omega, X/K}$, then the result can be obtained by generalized to the colimts and completion.
    \end{proof}
\end{lemma}
\begin{lemma}\label{pullback of local skeleton}
    For the canonical morphism $g: X_{\widehat{X^{\textsf{alg}}}} \longrightarrow X$, we have $g^{-1}(\textsf{M}_{\omega}) = \textsf{M}_{g^{*}\omega}$ for any $\omega \in \Gamma(X,\mathscr{K}_{X/K})$
    \begin{proof}
        The result can be obtained directly by Lemma \ref{pullback of Kahler seminorm} and the surjectivity  of $g$. 
    \end{proof}
\end{lemma}


 \begin{definition}\label{Kahsk defin}
     Let $f: X \longrightarrow S$ be a quasi-smooth morphism of $K$-analytic spaces. The \textit{Relative K\"ahler skeleton} of $f$ is defined as $\textsf{Sk}^{\text{k\"ah}}(X/S)\defeq\bigcup_{\omega\neq 0}\textsf{M}_{\omega}$. 
     For $S=\mathcal{M}(K)$, we denote $\textsf{Sk}^{\textsf{k\"ah}}(X)$ as the \textit{Absolute K\"ahler skeleton}. 
 \end{definition}
 For any open virtual disk $\D_{K}$, we have the following proposition:
 \begin{proposition}\label{sk-disk-(0,0)}
     ${\textsf{Sk}}^{\textsf{k\"ah}}(\D_{K})=\varnothing$. 
     \begin{proof}
     To see this, it suffices to show the result for $\D(0,1)_{\widehat{K^{{\textsf{alg}}}}}$ by Lemma \ref{pullback of Kahler seminorm} and Lemma \ref{pullback of local skeleton}. Then we need to note that we can get two closed disks $\mathbb{E}(0,r)_{\widehat{K^{{\textsf{alg}}}}} \subsetneq \mathbb{E}(0,r')_{\widehat{K^{{\textsf{alg}}}}} \subsetneq \D(0,1)_{\widehat{K^{{\textsf{alg}}}}}$ for $0 <r<r'<1$, assume $x_{0} \in {\textsf{Sk}}^{\textsf{k\"ah}}(\D(0,1)_{\widehat{K^{{\textsf{alg}}}}}) \cap \mathbb{E}(0,r)_{\widehat{K^{{\textsf{alg}}}}}$, then by Lemma \ref{etale-sk}, $x_{0}$ is a maximal point of both $\mathbb{E}(0,r)_{\widehat{K^{{\textsf{alg}}}}}$ and $\mathbb{E}(0,r')_{\widehat{K^{{\textsf{alg}}}}}$ which is impossible.
    \end{proof}
 \end{proposition}
 \begin{lemma}\label{ess-sk-curve}
     Let $X$ be an irreducible quasi-smooth analytic curve over $K$ with equal characteristic $0$, then ${\textsf{Sk}}^{\textsf{k\"ah}}(X)$ consists of points that do not admit an open neighborhood which is isomorphic to a virtual disk $\D_{K}$.
     \begin{proof}
         Let $x_{0} \in {\textsf{M}}_{\omega_{0}}$ for $\omega_{0} \neq 0$, then by Lemma \ref{etale-sk} $x_{0} \in {\textsf{M}}_{\tau^{*}\omega_{0}}$ where $\tau: \D_{K} \hookrightarrow X$, note that in this case $\tau^{*}\omega_{0} \neq 0$ and ${\textsf{Sk}}^{\textsf{k\"ah}}(\D_{K})=\varnothing$ by Proposition \ref{sk-disk-(0,0)}.
     \end{proof}
  \end{lemma}

\subsection{Relation with Essential Skeleton}
In this subsection, we assume $K$ is a complete discrete valued field with equal characteristic $0$. 
\par
The relation between the K\"ahler norm and the weight norm of Musta\c{t}\u{a}–Nicaise is established by M. Temkin by the following theorem in \cite{MR3702313}.

\par
        \begin{theorem}\cite[Theorem 8.3.3.]{MR3702313}
            If \,$K$ is discretely valued, $X$ is quasi-smooth and $x\in X$ is a divisorial point, i.e. a monomial point with discretely valued $\mathscr{H}(x)$, then the K\"ahler and the weight norms on $m$-canonical forms are related by:
            \begin{equation*}
                \| - \|_{\textsf{wt}^{\otimes m},x} = |\varpi_{K}|^{m} (\delta^{\textsf{log}}_{\mathscr{H}(x)|K})^{m} \| - \|_{\mathscr{K}_{X|K}^{\otimes m},x}
            \end{equation*}
        \end{theorem}
\begin{remark}\label{eqcharzero compare}
    When is a fixed non-Archimedean field with equal characteristic $0$, $\delta^{\textsf{log}}_{\mathscr{H}(x)|K}=1$, in this case, weight norm of Musta\c{t}\u{a}–Nicaise coincides with  K\"ahler norm.
\end{remark}
\begin{theorem}[M.Temkin]
    Let $\mathcal{X}$ be a smooth proper integral variety over a complete discrete valued field $K$ with equal characteristic $0$, then the associated essential skeleton $\textsf{Sk}^{\textsf{ess}}(\mathcal{
    X}^{\textsf{an}})$ coincides with K\"ahler skeleton $\textsf{Sk}^{\textsf{k\"ah}}(\mathcal{X}^{\textsf{an}})$.
    \begin{proof}
        This is a direct consequence following \cite[Theorem 8.2.9.]{MR3702313}, \cite[Theorem 8.3.3.]{MR3702313} and Remark \ref{eqcharzero compare}.
    \end{proof}
\end{theorem}
Now we demonstrate that no point of the essential skeleton may be contained in an open virtual disk:\begin{situation}\textit{Disk Interpretation of Essential Skeleton}.\label{Disk interpretation of essential skeleton}
\par
    Let $X$ be a $K$-analytic space of dimension $n$, $L|K$ be a non-Archimedean field extension of $K$ such that $d_{L|K} \leq \textsf{tr.deg}(L|K)=n-1$  and an open virtual disk $\D_{L}$ over $L$ which is contained in $X$, then there exist a family of elements $\{f_{i}\}_{1 \leq i \leq n-1}$ in $L$ such that for any $x \in \D_{L}$:
    \begin{itemize}
        \item $\{f_{i}\}_{1 \leq i \leq n-1}$ are independent with $x$.
        \item $\{f_{i}\}_{1 \leq i \leq n-1} \subseteq L \subseteq \kappa(x) \subseteq \mathscr{H}(x)$ are transcendental over $K$. 
        \item The images of $\{f_{i}\}_{1 \leq i \leq s_{L|K}}$ in $\widetilde{L} \subseteq \hcr $ are transcendental over $\widetilde{K}$. 
    \end{itemize}
    
    \par
    
Then there exists an open neighborhood $U_{x}$ of $x$ and a morphism of relative dimension $1$ by
\begin{equation*}
    \begin{tikzcd}[column sep=4em]
        \varphi_{x}: U_{x} \arrow[r,"(f_{1}\text{,}\dots\text{,}f_{n-1})"] & \G^{n-1,\textsf{an}}_{m, K}
    \end{tikzcd}
\end{equation*}
Now it's easy to see that we can glue $\varphi_{x}$ on $U_{L} \defeq \bigcup U_{x}$ as a smooth morphism $\varphi_{L}: U_{L} \longrightarrow \G^{n-1,\textsf{an}}_{m,K}$ such that $\D_{L} \subseteq \varphi_{L}^{-1}(L,|-|_{L}) \defeq C_{L}$ as an open subset.
\begin{theorem}\label{points in the essential skeleton}
 Let $\mathcal{X}$ be a smooth proper integral variety with dimension $n$ over $K$, then for any point $x$ of the  essential skeleton $\textsf{Sk}^{\textsf{ess}}(\mathcal{X}^{\textsf{an}})$ there exists no open virtual disk in $\mathcal{X}$ containing $x$.   
 \begin{proof}
     Let $x_{0} \in \textsf{M}_{\phi_{0}} \subseteq \textsf{Sk}^{\textsf{ess}}(\mathcal{X}^{\textsf{an}})$ such that $x_{0} \in \D_{L} $, and $x_{0} \in U_{L}$, now by Lemma \ref{etale-sk}, we have $x_{0} \in \textsf{M}_{\phi_{0}|_{U_{L}}} \subseteq \textsf{Sk}^{\textsf{k\"ah}}(U_{L}/K)$. We claim that $x_{0} \in \textsf{Sk}^{\textsf{k\"ah}}(U_{L}/\mathbb{G}_{m,K}^{n-1,\textsf{an}})$, to see this, let's denote the non-zero section $\omega_{0} \defeq \phi_{0}|_{U_{L}} \in \Gamma(U_{L},\mathscr{K}_{U_{L}/K})$ and then  consider the exact sequence of locally free $\mathscr{O}_{U_{L}}$-sheaves of finite rank on $U_{L}$ \cite[3.5.4. Corollary and 3.5.6. Corollary]{MR1259429} as the following:
     \begin{equation}
         \begin{tikzcd}
             0 \arrow[r] & \varphi_{L}^{*}\Omega_{\mathbb{G}_{m,K}^{n-1,\textsf{an}}/K} \arrow[r] & \Omega_{U_{L}/K} \arrow[r] & \Omega_{U_{L}/\mathbb{G}_{m,K}^{n-1,\textsf{an}}} \arrow[r] & 0
         \end{tikzcd}
         \end{equation}
    We have $\mathscr{K}_{U_{L}/K} \cong \varphi^{*}_{L}\mathscr{K}_{\G_{m,K}^{n-1,\textsf{an}}} \otimes_{\mathscr{O}_{U_{L}}} \mathscr{K}_{U_{L}/\G^{n-1,\textsf{an}}_{m,K}}$. Now take $\{U^{i}_{L}\}_{i}$ to be a covering of affinoid subdomains such that the restrictions of $\Omega_{U_{L}/\G^{n-1,\textsf{an}}_{m,K}}$ on $\{U_{L}^{i}\}_{i}$ are free, then 
    
\begin{align*}
\mathscr{K}_{U_{L}/K}(U_{L}^{i}) & \cong \varphi^{*}_{L}\mathscr{K}_{\G^{n-1,\textsf{an}}_{m,K}}(U_{L}^{i}) \otimes_{\Gamma(U^{i}_{L})} \mathscr{K}_{U_{L}/\G^{n-1,\textsf{an}}_{m,K}}(U_{L}^{i}) \\ & \cong  \varinjlim_{\varphi_{L}(U_{L}^{i}) \subseteq V_{i}}\mathscr{K}_{\G^{n-1,\textsf{an}}_{m,K}}(V_{i})\otimes_{\Gamma(U^{i}_{L})} \mathscr{K}_{U_{L}/\G^{n-1,\textsf{an}}_{m,K}}(U_{L}^{i})
\end{align*}
    We can now write $\omega_{0}|_{U_{L}^{i}}=[\frac{dt_{1}}{t_{1}} \wedge \cdots \wedge \frac{dt_{n-1}}{t_{n-1}}] \otimes \phi_{i}$. By local calculation, we have $\phi_{i}|_{U_{L}^{i} \cap U^{j}_{L}}=\phi_{j}|_{U_{L}^{i} \cap U^{j}_{L}}$. Let $\omega_{1}$ be the non-zero section of $\Gamma(U_{L},\mathscr{K}_{U_{L}/\G^{n-1,\textsf{an}}_{m,K}})$ by gluing $\{\phi_{i}\}_{i}$, we claim that $x_{0} \in \textsf{M}_{\omega_1}$. To see this, note that we have $\lvert\lvert \omega_{0}\lvert\lvert_{\mathscr{K}_{U_{L}/K},x_{0}}=\lvert\lvert \phi_{i} \lvert\lvert_{\mathscr{K}_{U_{L}/\G^{n-1,\textsf{an}}_{m,K}},x_{0}}$ by assuming $x_{0} \in U^{i}_{L}$, thus by \cite[Corollary 6.3.12.]{MR3702313}, $x_{0} \in \textsf{M}_{\omega_{1}|_{C_{L}}} \subseteq \textsf{Sk}^{\textsf{k\"ah}}(C_{L})$, but this contradicts Lemma \ref{ess-sk-curve}. 
 \end{proof}
\end{theorem}
\end{situation}
\section{Applications}

We have established that points outside the Berkovich skeleton of a strictly semistable model can be covered by open virtual disks, whereas no point of the essential skeleton can be contained in an open virtual disk. In this section, we assume $K$ is a complete discrete valued field with equal characteristic $0$.

\begin{conjecture}\label{bigconjecture}
 Let $\mathcal{X}$ be a smooth proper integral variety with dimension $n$ over $K$, then for any point $x \in \mathcal{X}^{\textsf{an}}$, $x$ is contained in an open virtual disk if and only if $x$ is not contained in the essential skeleton.
\end{conjecture}

We can prove this conjecture when we have a strcitly semistable model whose Berkovich skeleton is identified with the essential skeleton.

\begin{theorem}\label{semiample canonical}
Let ${{\mathcal{X}}}$ be a proper and strictly semi-stable scheme over complete discrete valuation ring $K^{\circ}$ such that the canonical divisor $K_\mathcal{X}$ is semiample. Then for any point $x \in \mathcal{X}_{K}^{\textsf{an}}$, $x$ is contained in an open virtual disk if and only if $x$ is not contained in the essential skeleton.
\end{theorem}
\begin{proof}
The condition that the canonical divisor is semiample on our strictly semi-stable model implies that the Berkovich skeleton of the model and the essential skeleton are identical, by \cite[Theorem 3.3.3]{MR3595497}. The result then follows from combining Theorem \ref{points in the essential skeleton} and Theorem \ref{main theorem 1}.
\end{proof}
In the case where $\mathcal{X}$ admits no global pluricanonical forms we expect by Mori's conjecture that $\mathcal{X}$ is uniruled. In this case it is natural to say the essential skeleton is empty. We wish to formulate the non-Archimedean analog interpretation of uniruled variety over a complete discrete valued field  $K$ with equal characteristic $0$.

Let's first recall the classical setting. A proper variety $\mathcal{X}$ is called \textit{uniruled} if there is a dominant rational map $\varphi: \mathcal{Y} \times \mathbb{P}^{1} \dashrightarrow \mathcal{X}$,
 where $\mathcal{Y}$ is a variety of dimension $\textsf{dim}(\mathcal{X})-1$. $\mathcal{X}$ being uniruled obviously
implies that rational curves cover $\mathcal{X}$. On the other hand, if $\mathcal{X}$ is smooth and uniruled, then on $\mathcal{X}$, there cannot exist any pluricanonical form, thus by definition, the essential skeleton $\textsf{Sk}^{\textsf{ess}}(\mathcal{X}^{\textsf{an}})=\emptyset $. Motivated by Theorem \ref{points in the essential skeleton}, we propose the following definition and conjecture:
\begin{definition}
    Let $X$ be a proper geometrically integral smooth $K$-analytic space, $X$ is called \textit{non-Archimedean uniruled} if $X^{}$ can be covered by open virtual disks after finite field extension $K'|K$ for $X$.
\end{definition}
\begin{conjecture}\label{main conjecture uniruled} Let $\mathcal{X}$ be a uniruled variety over $K$, then $\mathcal{X}^{\textsf{an}}$ is non-Archimedean uniruled. 
    
\end{conjecture}

Conjecture $\ref{main conjecture uniruled}$ is the special case of Conjecture $\ref{bigconjecture}$ when $\mathcal{X}$ is uniruled and therefore has empty essential skeleton.
\begin{example}
    The projective line $\mathbb{P}^{1}$ is non-Archimedean uniruled. 
    To see this, just note that $\mathbb{P}^{1,\textsf{an}}=\textsf{trop}^{-1}[0,r_{1}) \cup \textsf{trop}^{-1}(r_{2},\infty]$ with $r_{1} > r_{2}$.
\end{example}
\begin{example}
    Any ruled variety $\mathcal{X}$ is non-Archimedean uniruled, since $\mathcal{X}=\mathcal{Y} \times \mathbb{P}^{1}$, then $\mathcal{X}^{\textsf{an}} \cong \mathcal{Y}^{\textsf{an}} \times \mathbb{P}^{1,\textsf{an}}\cong \bigcup_{y \in \mathcal{Y}^{\textsf{an}}} \mathbb{P}_{\mathscr{H}(y)}^{1,\textsf{an}}$
\end{example}
\begin{situation}
    \textit{Being Non-Archimedean uniruled is invariant under blow-ups.}. \par
Let us consider a non-Archimedean analytification of a  blow-up $\pi: \widetilde{\mathcal{X}} \defeq \textsf{Bl}_{\mathcal{Z}}\mathcal{X} \longrightarrow \mathcal{X}$ where $\mathcal{Z}$ is a smooth closed subscheme of $\mathcal{X}$, then we have the following result:

\color{black}
    \begin{proposition}\label{NA uniruled blowup}
        Let $\mathcal{X}$ be a smooth geometrically integral proper variety over $K$, then  $\mathcal{X}^{\textsf{an}}$ is non-Archimedean uniruled if and only if $\widetilde{\mathcal{X}}^{\textsf{an}}$ is non-Archimedean uniruled.
        \begin{proof}
            To prove sufficiency, let us assume $\mathcal{X}$ is non-Archimedean uniruled and $K''|K$ be a finite field extension such that $\mathcal{X}_{K''}^{\textsf{an}}$ can be covered by open virtual disks and $\widetilde{\mathcal{X}}_{K''}^{\textsf{}}$ admits a proper strictly semi-stable model. By \cite[Proposition 13.91.]{MR4225278}, we have $\mathcal{X}_{K''}^{\textsf{an}}\setminus \mathcal{Z}_{K''}^{\textsf{an}}=\bigcup \D_{L} \setminus \mathcal{Z}_{K''}^{\textsf{an}} \cong \widetilde{\mathcal{X}}_{K''}^{\textsf{an}} \setminus \pi_{K''}^{\textsf{an},-1}(\mathcal{{Z}_{K''}^{\textsf{an}}})  $, meanwhile, by \cite[Proposition 2.7.7.]{MR3826929} we have the fiber $\widetilde{\mathcal{X}}^{\textsf{an}}_{K'',x} \cong \mathbb{P}^{r-1,\textsf{an}}_{\mathscr{H}(x)}$ for $r=\dimm_{x}(\widetilde{\mathcal{X}}_{K''}^{\textsf{an}})-\dimm_{x}(\mathcal{X}_{K''}^{\textsf{an}}) $ and $x \in \mathcal{Z}_{K''}^{\textsf{an}}$, thus $x \in \widetilde{\mathcal{X}}^{\textsf{an}}_{K'', x}$.  Note that $\pi_{K''}^{\textsf{an},-1}(\mathcal{Z}_{K''}^{\textsf{an}})$ has no intersection with any Berkovich skeleton $\textsf{Sk}(\widetilde{\mathfrak{X}})$ for a fixed strictly semi-stable model $\widetilde{\mathfrak{X}}$ of $\widetilde{\mathcal{X}}_{K''}$, thus by Theorem \ref{main theorem 1}, we have $\widetilde{\mathcal{X}}_{K''}^{\textsf{an}}=\bigcup \D_{L}$. For the necessity, let's assume there exists a point $x_{0} \in \mathcal{X}^{\textsf{an}}$ such that any virtual disks cannot cover it after any finite field extension of $K$, for a sufficient large finite field extension $K'|K$, by Theorem \ref{main theorem 1} again, we can assume $x_{0}$ is a birational point, then by assumption, we have $x_{0} \in \mathbb{D}_{L} \subseteq \widetilde{\mathcal{X}}_{K'}^{\textsf{an}}$ and $x_{0} \notin \pi_{K'}^{\textsf{an},-1}(\mathcal{Z}_{K'}^{\textsf{an}})$ for $L|K'$ be a non-Archimedean field extension with transcendental degree $n-1$, thus we can choose $\mathbb{D}_{L} \cap \pi_{K'}^{\textsf{an},-1}(\mathcal{Z}_{K'}^{\textsf{an}}) = \emptyset$ which leads $x_{0} \in \mathbb{D}_{L} \subseteq \mathcal{X}_{K'}^{\textsf{an}}$, contradiction. 
        \end{proof}
    \end{proposition}

    \begin{example}\label{Pn is NA uniruled}
        $\mathbb{P}_{K}^{n}$ is non-Archimedean uniruled.
      To see this, let's
            recall that the moduli space of stable genus zero curves with $n$-marked points $\overline{\textsf{M}}_{0,n}$ can be constructed as a smooth iterated  blow-ups of the Cartesian power $(\mathbb{P}_{K}^{1})^{n-3}$ \cite[Theorems 1 and 2]{MR1034665} and a smooth iterated blow-up of $\mathbb{P}_{K}^{n-3}$ \cite{MR1237834}. On the other hand $(\mathbb{P}_{K}^{1})^{n-3,\textsf{an}} \cong (\mathbb{P}^{1,\textsf{an}}_{K})^{n-3}$ is non-Archimedean uniruled, thus by Lemma \ref{NA uniruled blowup}, we are done.

    \end{example}
\end{situation}

 \section{Complement of the Skeleton via Analytic Functions}
 In the appendix, we first review the naive analytic functions sheaf developed by V. Berkovich \cite{MR2263704}. We give an alternative description of the complement of Berkovich skeleton in terms of analyticity sets of Iwasawa log functions.

\begin{situation}
    Let $X$ be a $K$-analytic space. For $\mathscr{F} \in \textsf{Sh}(X_{\textsf{\'et}})$ an \'etale sheaf, we define an associated presheaf $\widetilde{\mathscr{F}}$ as $\widetilde{\mathscr{F}}(Y) \defeq \varinjlim \mathscr{F}(V)$, where V are open neighborhoods of $Y_{\textsf{st}}$ in $Y$.
\end{situation}
\begin{situation}
    Let $L$ be a filtered $K$-algebra, i.e., a commutative $K$-algebra provided by an increasing sequence of $K$-vector subspaces $L^{0} \subseteq L^{1} \subseteq \cdots$ such that $L^{i}\cdot L^{j} \subseteq L^{i+j}$ and $L=\bigcup^{\infty}_{i=0}L^{i}$. For a $K$-analytic space $X$, we denote $\mathscr{O}_{X}^{L,i}\defeq \mathscr{O}_{X} {\otimes}_{K} L^{i}$. 
    The sheaf of $\mathfrak{N}^{L,i}$-analytic functions $\mathfrak{N}_{X}^{L,i}$ is the associated sheaf $\widetilde{\mathscr{O}_{X}^{L,i}}$ and we define the sheaf of $\mathfrak{N}^{L}$-analytic function $\mathfrak{N}_{X}^{L}$ as $\mathfrak{N}_{X}^{L}\defeq \varinjlim \mathfrak{N}_{X}^{L,i}$ as a sheaf of filtered $\mathscr{O}_{X}$-algebra. For $f \in \mathfrak{N}_{X}^{L}(X)$, the \textit{analyticity set} of $f$ is the associated unique maximal open set $U_{f}$ such that $f$ comes from $\mathscr{O}_{X}^{L}(X)$ and $X_{\textsf{st}} \subseteq U_{f} \subseteq X $. When we take $L=K$, we denote $\mathfrak{N}_{X}^{K}$ by $\mathfrak{n}_{X}$.
\end{situation}
\begin{situation}
    \textsc{Pull back of na\"ive analytic functions sheaf $\mathfrak{N}^{L}$} \par
    Consider a pair $(f,g)$, where $f$ is a morphism of $K$-analytic spaces $X \longrightarrow Y$ and $g$ is a morphism of filtered algebra $L \longrightarrow S$ over a field extension $K \subseteq K' \subseteq \widehat{K^{\textsf{alg}}}$, then we can define the canonical morphism $f^{\sharp}: f^{*}\mathfrak{N}_{Y}^{S} \longrightarrow \mathfrak{N}^{L}_{X}$. To see this, it suffices to show there exist well-defined morphisms $f_{i}^{\flat}: \mathfrak{N}^{S_{i}}_{Y} \longrightarrow f_{*}\mathfrak{N}^{L_{i}}_{X}$ for any $i \ge 0$, now for $V \in \textsf{\'Et}(Y)$ and $\varphi \in \mathfrak{N}^{S_{i}}_{Y}(V)$, assume $\varphi \in \mathscr{O}^{S_{i}}_{Y}(U_{\varphi})$ where $V_{\textsf{st}} \subseteq U_{\varphi} \subseteq V$, then we have $f^{\flat}(\varphi) \in \mathscr{O}^{L_{i}}_{X}(U_{\varphi} \times_{Y} X)$, meanwhile we have $f(X_{\textsf{st}}) \subseteq Y_{\textsf{st}}$ and $\big(V \times_{Y}X\big)_{\textsf{st}}=\textsf{fp}_{V}^{-1}(V_{\textsf{st}})$ by Proposition \ref{points under \'etale morphisms}, thus $\big(V \times_{Y} X\big)_{\textsf{st}} \subseteq U_{\varphi} \times_{Y} X$ and $f_{i}^{\flat}(\varphi) \in f_{*}\mathfrak{N}^{L_{i}}_{X}(V)$, we will denote $f_{i}^{\flat}(\varphi)$ as $f^{*}(\varphi)$ in the following texts. We should note that $U_{\varphi} \times_{Y} X \subsetneq U_{f^{*}(\varphi)}$ in general.
    \par

    Now consider the canonical morphism of \'etale abelian sheaves $\textsf{Log}^{\lambda}: \mathbb{G}^{1,\textsf{an}}_{m,X} \longrightarrow \mathfrak{N}^{L,1}_{X}$ induced by a fixed logaritmic character $\lambda$. For $U \in \textsf{\'Et}(X)$, $f \in \mathbb{G}^{1,\textsf{an}}_{m,X}(U)=\mathscr{O}(U)^{\times}$, then $\textsf{Log}^{\lambda}(f) \in \mathfrak{N}_{X}^{L,1}(U)$ as the construction above.

\end{situation}
\begin{situation}
    Recall that \textit{a branch of the logarithm} over $L$ is an $\mathfrak{N}^{L,1}$-analytic function $f$ on $\mathbb{G}^{1,\textsf{an}}_{m,K}$ with coordinate $T$ such that $df=\frac{dT}{T}$. Morover, if we consider the three maps $\textsf{m},\textsf{fp}_{1},\textsf{fp}_{2}:\mathbb{G}^{2,\textsf{an}}_{m,K} \rightarrow \mathbb{G}^{1,\textsf{an}}_{m,K}$, we should have $\textsf{m}^{*}(f)=\textsf{fp}_{1}^{*}(f)+\textsf{fp}^{*}_{2}(f)$. Now note that $\mathbb{G}^{1,\textsf{an}}_{m,K}(K)=\mathbb{G}^{1}_{m,K}(K)=K^{\times}$ is contained inside the analytic set $U_{f}$ of $f$, then $f(a \cdot b)=\textsf{m}^{*}(f)(a, b)=\textsf{fp}^{*}_{1}(f)(a, b)+\textsf{fp}^{*}_{2}(f)(a, b)=f(a)+f(b)$ and $f \in \Gamma(\mathscr{O}^{L,1}_{\mathbb{G}^{1,\textsf{an}}_{m,K}},U_{f})$, such an $f$ gives a $\textsf{Gal}(K^{\textsf{alg}}|K)$-homomorphism $\lambda_{f}: (K^{\textsf{alg}})^{\times} \longrightarrow {(K^{\textsf{alg}})^{\times}}\otimes_{K}L^{1}$ of abelian groups which is called the \textit{logarithmic character with values in} $L$. By \cite[Lemma 1.4.1.]{MR2263704}, there is a one-to-one correspondence between $\lambda_{f}$ and $f$ which we denote by $[\textsf{Log}^{\lambda_{f}}]$, and for any $\lambda$, $\textsf{Log}^{\lambda} \in \mathscr{O}(U_{\textsf{log}^{\lambda}})$ where $U_{\textsf{log}^{\lambda}}=\bigcup_{x \in (\mathbb{G}^{1,\textsf{an}}_{m,K})_{0}} \mathbb{D}(x,|T|_{x})$, it's easy to see that $\textsf{Log}^{\lambda}$ is well defined, for any $\mathbb{D}(b_{1},|b_{1}|) \cap \mathbb{D}(b_{2},|b_{2}|) \neq \varnothing$, we have $|b_1|=|b_2|$ ,then 
    \begin{align*}
        &\,\,\,\,\,\,\, \lambda(b_{1})+\textsf{Log}(\frac{T}{b_{1}})-\lambda(b_{2})-\textsf{Log}(\frac{T}{b_{2}})\\ &= \lambda(\frac{b_1}{b_2})+\textsf{Log}(\frac{b_2}{b_1})=0
        \end{align*} since $|\frac{b_{1}}{b_2}|=1$.
        So $U_{\textsf{Log}^{\lambda}}$ is the maximal analytic continuation of $\textsf{Log}$.
        \begin{proposition}\cite[Lemma 1.4.1.]{MR2263704}
            The analyticity set of any branch of Iwasawa log function $\logg^{\lambda}$ on $\mathbb{G}^{1,\textsf{an}}_{m,K}$ is equal to the complement of the skeleton $\textsf{Sk}(\mathbb{G}^{1,\textsf{an}}_{m,K})$.
        \end{proposition}
    \par
    \begin{definition}
        Let $X$ be a quasi-smooth $K$-analytic space, then for any morphism $f: X \longrightarrow \mathbb{G}^{1,\textsf{an}}_{m,K}$, we can get a naive analytic function $\logg^{\lambda}(f) \defeq f^{*}\logg^{\lambda}$ on $X$ which is called \textit{Iwasawa log function} on $X$. 
    \end{definition}
    Consider $\textsf{fp}_{i}: \mathbb{G}^{n,\textsf{an}}_{m,K} \longrightarrow \mathbb{G}^{1,\textsf{an}}_{m,K}$, take $\textsf{Log}^{\lambda} \in \mathfrak{N}^{K,{1}}_{\mathbb{G}^{1,\textsf{an}}_{m,K}}(\mathbb{G}^{1,\textsf{an}}_{m,K})$, we have $\textsf{fp}^{*}_{i}(\textsf{Log}^{\lambda})=\textsf{Log}^{\lambda}(T_{i}) \in \mathfrak{N}^{K,{1}}_{\mathbb{G}^{n,\textsf{an}}_{m,K}}(\mathbb{G}^{n,\textsf{an}}_{m,K})$

    \begin{proposition}\label{complement of skeleton of Gn}
    The union of analyticity set of the na\"ive analytic function $\textsf{Log}^{\lambda}(T_{i})$ is the complement of the skeleton of $\mathbb{G}^{n,\textsf{an}}_{m,K}$
    \begin{proof}
        Note that the analyticity set $U_{\textsf{Log}^{\lambda}(T_{i})}$ of $\textsf{Log}^{\lambda}(T_{i})$ is $$\mathbb{G}^{1,\textsf{an}}_{m,K} {\times}_{K}\mathbb{G}^{1,\textsf{an}}_{m,K} \cdots \mathbb{G}^{1,\textsf{an}}_{m,K}{\times}_{K}U_{\textsf{Log}^{\lambda}}{\times}_{K} \mathbb{G}^{1,\textsf{an}}_{m,K} \cdots {\times}_{K} \mathbb{G}^{1,\textsf{an}}_{m,K} ,$$ then by Lemma \ref{the complement of skeleton of torus}, we have $\textsf{Sk}(\mathbb{G}^{n,\textsf{an}}_{m,K})^{\textsf{c}}=\bigcup_{\alpha} U_{\textsf{Log}^{\lambda}(T_{i})}$.
    \end{proof}
\end{proposition}
\begin{proposition}\label{complement of skeleton of mono chart}
    For the monomial chart  $\varphi_{M}: M \hookrightarrow \mathbb{G}^{n,\textsf{an}}_{m,K}$ in \ref{mono-chart-semistable}, then $\textsf{Sk}(\widehat{\mathcal{M}}_{\eta})^{\textsf{c}}$ the complement of skeleton $\textsf{Sk}(\widehat{\mathcal{M}}_{\eta})$ is $\bigcup U_{\varphi^{*}_{M}(\logg^{\lambda}(T_{i}))}$.
    \begin{proof}
        By \ref{complement of skeleton of Gn}, it suffices to show $\varphi^{-1}_{M}(U_{\logg^{\lambda}(T_{i})})=U_{\varphi_{M}^{*}(\logg^{\lambda}(T_{i}))}$. Without loss of generality, we show the case for $i=1,n=2$, then $U_{\logg^{\lambda}(T_{1})}=U_{\logg^{\lambda}} \times_{K} \mathbb{G}^{1,\textsf{an}}_{m,K}=\bigcup U_{\logg^{\lambda},\mathscr{H}(x)}$ which could reduce the case to show $\varphi^{-1}_{M,x}(U_{\logg^{\lambda},\mathscr{H}(x)})=U_{\varphi^{*}_{M,x}(\logg^{\lambda}),\mathscr{H}(x)}$. Meanwhile by Lemma \ref{mono-chart-semistable}, analytic locally, we have $M_{x} \cong \mathbb{B}(0;r,1)_{\mathscr{H}(x)}$ for $r \in |\mathscr{H}(x)^{\times}|$ which is an analytic domain $\{x \in \mathbb{G}^{1,\textsf{an}}_{m,K}\,|\, r \leq |T|_{x} \leq 1\}$.
    \end{proof}
    \end{proposition}
\begin{proposition}\label{log description of complement of semistable skeleton}
    Let $X$ be a scheme that admits semi-stable reduction, the complement of its Berkovich skeleton is a union of analyticity sets of the naive analytic function $\logg^{\lambda}$.
    \begin{proof}
        By Proposition \ref{complement of skeleton of mono chart}, it suffices to show that for \'etale morphism $g: \mathfrak{V}_{\eta} \longrightarrow \mathbb{B}(0;r_{1},r_{2})_{K} \subseteq \G^{1,\textsf{an}}_{m,K}$, we have $g^{-1}(U_{\logg^{\lambda}})=U_{\logg^{\lambda}(g)}$. Denote $D' \defeq \{x \in \mathfrak{V}_{\eta}\,|\,|g(x)-1| < 1\}$, then we have $\logg^{\lambda}(g)|_{D'}=\logg(g)$. Assume $U_{\logg^{\lambda}(g)}\neq g^{-1}(U_{\logg^{\lambda}})$, then there is a connected open subset $U'$ of $\mathfrak{V}_{\eta}  $ which strictly contains $D'$ and $\logg^{\lambda}(g)|_{U'} \in \mathscr{O}(U')$, which is impossible by the same argument of \cite[Lemma 1.4.1]{MR2263704}.
    \end{proof}
\end{proposition}
\end{situation}

\appendix 


\bibliographystyle{amsalpha}
\bibliography{main}
\printindex

\end{document}